\documentclass[11pt,reqno]{amsart}
\usepackage{cmap,mathtools}
\usepackage[T1]{fontenc}
\usepackage[utf8]{inputenc}
\usepackage[english]{babel}
\usepackage{amsfonts,amssymb,amsmath}
\usepackage{footnote,bigints}
\usepackage{enumitem}

\allowdisplaybreaks
\usepackage[pagewise, mathlines]{lineno}
\usepackage{graphicx}
\usepackage{epstopdf}
\usepackage{a4wide}
\usepackage{xcolor}
\usepackage[colorlinks=true,linktocpage,pdfpagelabels,
bookmarksnumbered,bookmarksopen]{hyperref}
\definecolor{ForestGreen}{rgb}{0.1,0.6,0.05}
\definecolor{EgyptBlue}{rgb}{0.063,0.1,0.6}
\hypersetup{
	colorlinks=true,
	linkcolor=EgyptBlue,         
	citecolor=ForestGreen,
	urlcolor=olive
}

\usepackage[hyperpageref]{backref}

\newtheorem{theorem}{Theorem}

\newtheorem{lem}[theorem]{Lemma}

\theoremstyle{definition}

\newtheorem{rmk}[theorem]{Remark}

\let\OLDthebibliography\thebibliography
\renewcommand\thebibliography[1]{
	\OLDthebibliography{#1}
	\setlength{\parskip}{1pt}
	\setlength{\itemsep}{1pt plus 0.3ex}
}

\numberwithin{equation}{section}
\numberwithin{theorem}{section}
\numberwithin{equation}{section}
\numberwithin{theorem}{section}

\usepackage{ulem}
\usepackage{xcolor}
\usepackage[colorlinks=true,linktocpage,pdfpagelabels,
bookmarksnumbered,bookmarksopen]{hyperref}
\definecolor{ForestGreen}{rgb}{0.1,0.6,0.05}
\definecolor{EgyptBlue}{rgb}{0.063,0.1,0.6}
\hypersetup{
	colorlinks=true,
	linkcolor=EgyptBlue,         
	citecolor=ForestGreen,
	urlcolor=olive
}

\usepackage[hyperpageref]{backref}
\usepackage[foot]{amsaddr}

\title [isoperimetric inequality in ROSS]{An isoperimetric inequality for the harmonic mean of Steklov eigenvalues in rank one symmetric spaces}

\author{Hemangi Madhusudan Shah$^1$ \and Sheela Verma$^2$}
\address{$1$ Harish-Chandra Research Institute, A CI of Homi Bhaba National Institute, Chhatnag Road, Jhunsi, Prayagraj-211019, India}
\email{hemangimshah@hri.res.in}
\address{$2$ Corresponding author, Department of Mathematical Sciences, Indian Institute of Technology(BHU), Varanasi}
\email{sheela.mat@iitbhu.ac.in}

\begin{document}
	
\begin{abstract}
In this note, an isoperimetric inequality for the harmonic mean of lower order Steklov eigenvalues is proved on bounded domains in noncompact rank-$1$ symmetric spaces. This work extends result of \cite{BR.01} and \cite{V.21} proved for bounded domains in Euclidean space and Hyperbolic space, respectively.
\end{abstract}
	
\maketitle 

\section{Introduction}
In spectral geometry, the problem of finding an optimal domain for Laplace eigenvalues or some combinations of Laplace eigenvalues is of wide interest. In this paper, we shall study the geometric behaviour of the Steklov eigenvalues on bounded domains in noncompact rank-$1$ symmetric spaces. On a compact, complete Riemannian manifold $M$ with smooth boundary $\partial M$, the Steklov eigenvalue problem is given by,
\begin{align*}
\begin{array}{lcl}
&\Delta u  & = 0\;\; \mbox{in} \; M \\
&\frac{\partial u}{\partial \nu} & = \mu u \;\; \mbox{on} \; \partial M,
\end{array}
\end{align*}
where $\displaystyle\frac{\partial u}{\partial \nu}$ denotes the  directional derivative of $u$ in the direction of the outward unit normal $\nu$ and $\mu$ is a real number which is called Steklov eigenvalue corresponding to eigenfunction $u$. It is well known that the Steklov eigenvalues are discrete and form an increasing sequence
$$ 0 = \mu_{0}(M) < \mu_{1}(M) \leq \mu_{2}(M) \leq \cdots \nearrow \infty.$$
Using a technique of conformal mapping, Weinstock \cite{W.54} proved that for a simply connected planar domain $\Omega$ with analytic boundary of perimeter $l$, 
 \begin{align} \label{Weinstock inequality}
\mu_{1}(\Omega) \leq \frac{2 \pi}{l}.
 \end{align}
 Hersch and Payne \cite{HP.68} noticed that the Weinstock’s proof provides a  more general result, namely
\begin{align*}
    \frac{1}{\mu_{1}(\Omega)} + \frac{1}{\mu_{2}(\Omega)} \geq \frac{l}{\pi}.
\end{align*}
\noindent
Later, F.Brock \cite{BR.01} extended this result for bounded Lipschitz domain $\Omega$ in $\mathbb{R}^{n}$ under the volume constraint and proved that
\begin{align}\label{Brock result}
   \sum_{i=1}^{n} \frac{1}{\mu_{i}(\Omega)}  \geq n \ R = \sum_{i=1}^{n} \frac{1}{\mu_{i}(B(R))},
\end{align}
where $B(R) \subset \mathbb{R}^{n}$ is a ball of radius $R$ having the same volume as $\Omega$. Recently, S.Verma generalised \eqref{Brock result} for the first $(n-1)$ nonzero Steklov eigenvalues on bounded domains in an $n$-dimensional hyperbolic space \cite{V.21}. In \cite{BS.14}, Binoy and Santhanam studied the Steklov eigenvalue problem on bounded domains of a noncompact rank-$1$ symmetric space and proved that among all the domains of fixed volume, the ball maximizes the first nonzero Steklov eigenvalue. For more results related to spectral geometry of rank-$1$ symmetric spaces, we refer to \cite{BS.14, CR.19,MW.22, S.07, VS.20}.\\

In this article, we extend the results of \cite{BR.01} and \cite{V.21} and prove an isoperimetric bound for the lower Steklov eigenvalues on convex domains in noncompact rank-$1$  symmetric spaces. More precisely, we prove the following theorem. 

\begin{theorem} \label{thm:main theorem}
Let ($M,ds^{2})$ be a noncompact rank-$1$ symmetric space with curvature $-4 \leq K_{M} \leq -1$ and $\Omega$ be a bounded convex domain in $M$ having smooth boundary $\partial \Omega$. Let $B(R) \subset M$ denote a geodesic ball of radius $R>0$ centered at $p$ such that vol$(\Omega)$ = vol$(B(R))$. Then
\begin{align*}
\sum_{i=1}^{l} \frac{1}{\mu_{i}(\Omega)} \geq \sum_{i=1}^{l} \frac{1}{\mu_{i}(B(R))},
\end{align*}
where
\begin{align*}
l=\begin{cases}
k-1 , & n=1 \\
k(n-1), & n>1.
\end{cases}
\end{align*}
Further, the equality holds if and only if $\Omega$ is a geodesic ball.
\end{theorem}

\section{Geometry of non-compact rank-$1$ symmetric spaces}
In this section, we recall some geometric properties of non-compact rank one symmetric spaces \cite{CR.19}. 
There are two families of simply connected rank one symmetric spaces which are dual to each other. First one is compact rank one symmetric spaces: round sphere ${\mathbb{S}}^n$,  complex projective space ${\mathbb{C}}{P}^n$,  quaternionic projective space ${\mathbb H}{P}^n$, and  Cayley projective plane $Ca{\mathbb{P}}^2$. Second one is non-compact rank one symmetric spaces: real hyperbolic space $\mathbb{R}{H}^n$, complex hyperbolic space ${\mathbb{C}}{H}^n$, quaternionic hyperbolic space ${\mathbb H}{H}^n$, and Cayley hyperbolic plane  $Ca H^2$.\\\\
\indent
In the sequel, ${\mathbb{K}}$ denotes $\mathbb{R}$ (the field of real numbers), ${\mathbb{C}}$ (the field of complex numbers), $\mathbb{H}$ (the algebra of quaternions) or $Ca$ ( the algebra of octonions). Let $k = \dim_{\mathbb R}K$, then  $m = kn$. The noncompact rank one symmetric spaces are defined as follows. 
On  ${\mathbb{K}}^{n+1}$ define the Hermitian inner product by,
$$ \langle x, y \rangle = - x_{0}{\bar y_{0}} + \sum_{i=1}^{n} x_{i}{\bar y_{i}},$$ whose real part $\langle.,. \rangle_{\mathbb{R}}$ is a real 
bilinear form with signature $(k, nk)$ on ${\mathbb{K}}^{n+1}$. 
Consider the Lorentzian sphere in ${\mathbb{K}}^{n+1}$ defined by,
$$H^{kn+k-1}= \{x \in {\mathbb{K}}^{n+1} | \langle x, x \rangle = -1\}.$$  Then $U(1, \mathbb{K})$, the set of unit elements acts on $H^{kn+k-1}$,
and the hyperbolic space  $\mathbb{K}H^{n}$ is 
$H^{kn+k-1}/ U(1, \mathbb{K})$. It is the base space of the fibration
$${\mathbb{S}}^{k-1} \rightarrow  H^{kn+k-1} \rightarrow \mathbb{K}H^{n}.$$
The metric induced by  $\langle.,. \rangle_{\mathbb{R}}$ on $H^{kn+k-1}$
has signature $(k-1, kn)$, and since it is preserved by the action 
$U(1, \mathbb{K})$ whose orbits are $k-1$ dimensional spheres, its restriction to the fibre in orthogonal direction is positive definite. 
Therefore, the fibration induces a Riemannian metric $\mathbb{K}H^{n}$. \\\\
Let $B(p,r) \subset M$, a noncompact rank one symmetric space, be the geodesic ball centered at $p$ having radius $r$. And $S(p,r) =\partial (B(p,r)),$ be the geodesic sphere of radius $r$. 
Let $A(r)$ denote the second fundamental form of $S(p,r)$. It is known that (page $8$ of \cite{CR.19}) the volume density function of $S(p,r)$ is given by:
\begin{align*}
\phi(p,r) = \sinh^{kn-1}(r) \cosh^{k-1}(r)
\end{align*}
and Tr$(A(r)) = \frac{\phi'(r)}{\phi(r)}$.
The first nonzero eigenvalue of the Laplacian $\lambda_{1}(S(p,r))$ on $S(p,r)$ has multiplicity $kn$ and is given by,
\begin{align*}
\lambda_{1}(S(p,r)) & =  ({\mbox{Tr}(A(r))})' =
\frac{kn-1}{\sinh^{2}(r)} - \frac{k-1}{\cosh^{2}(r)}, \\
& = \frac{kn-k}{\sinh^{2}(r)} + \frac{k-1}{\sinh^{2}(r) \cosh^{2}(r)}.
\end{align*}
Let $(x_{1}, x_{2}, \ldots, x_{kn})$ be a geodesic normal coordinate system in $B(p,r)$.  Then the coordinate functions $\displaystyle \frac{x_{i}}{r}, 1 \leq i \leq kn$ are the eigenfunctions corresponding to $\lambda_{1}(S(p,r))$. 

The Laplace operator on $M$ in polar coordinates can be written as:
$$\Delta_{M} = \frac{{\partial}^2}{\partial r^2} + \mbox{Tr} (A(r)) + \Delta_{S(p,r)}.$$
 Note that the aforementioned identities  
are independent of the base point $p$ chosen and hence it can be omitted. 

\section{Characterization of Steklov eigenvalues on Geodesic ball}
It is well known that (Proposition $3.7$ of 
\cite{CR.19}) the first nonzero Steklov eigenvalue $\mu_{1}(B(R))$ is of multiplicity $kn$ and given by,
\begin{align*}
\mu_{1}(B(R))&= \frac{\int_{B(R)} g^{2}(r) \lambda_{1}(S(r)) + (g'(r))^{2} dV}{\int_{S(R)} g^{2}(r) dV} \\
&= \frac{\int_{B(R)} g^{2}(r) \lambda_{1}(S(r)) + (g'(r))^{2} dV}{g^{2}(R) \mbox{vol}(S(R))},
\end{align*}
where $g(r)$ is the radial function satisfying:
\begin{align*}
\begin{array}{rcll}
g''(r) + \mbox{Tr}(A(r)) g'(r) - \lambda_{1}(S(r)) g(r) = 0, \ r \in (0,R), \\
g(0) = 0, \quad g'(R) = \mu_{1}(B(R)) g(R).
\end{array}
\end{align*}
The eigenfunctions corresponding to $\mu_{1}(B(R))$ are of the form $\displaystyle g(r) \frac{x_{i}}{r}, 1 \leq i \leq kn$, where $(x_{1}, x_{2}, \ldots, x_{kn})$ are geodesic normal coordinate system centered at point $p$. Using the fact that Tr$(A(r)) = \frac{\phi'(r)}{\phi(r)}$ and $(\text{Tr}(A(r)))' = - \lambda_{1}(S(r))$, we get
\begin{align} \label{function g}
g(r) = \frac{1}{\phi(r)} \int_{0}^{r}\phi(t) dt.
\end{align}
The following property of function $g$ is proved in Lemma $3.4$ of \cite{BS.14}.
\begin{lem}
$\left(g'(r)\right)^{2} + g^{2}(r) \lambda_{1}(S(r))$ is a decreasing function for $r>0$.
\end{lem}

\section{Construction of test functions for first $kn-$nonzero Steklov eigenvalues}
Let $(M,g)$ be a complete Riemannian manifold.
Let $\{u_{i}\}_{i=0}^{\infty}$ be an orthonormal sequence of eigenfunctions corresponding to the eigenvalues $\{\mu_{i}(\Omega)\}_{i=0}^{\infty}$. For $1 \leq i < \infty$, the variational characterization of $\mu_{i}(\Omega)$ is given by
 (see eg., Section $2$ of \cite{CR.19}),
\begin{align}\label{var-cha}
 \mu_{i}(\Omega) = \inf_{f \in H^{1}(\Omega) \setminus \{0\}} \left\{ \frac{\int_{\Omega} \|\nabla f \|^{2} dV}{\int_{\partial\Omega} f^{2} dA } : \int_{\partial\Omega} f u_{j} dA = 0, 0 \leq j \leq i-1 \right\}.
\end{align}

For a domain $A \subset M$, let $CA$ denote the convex hull of $A$. Let $(X_{1}, X_{2}, \ldots, X_{n})$ be a geodesic normal coordinate system centered at $p$. We identify $CA \subset M$ as its inverse image under $\exp_{p}^{-1}$ in $T_{p}M$ and write $ds_{p}^{2}(X,X)$ as $\|X\|_{p}^{2}$ for $X \in T_{p}(M)$. Let $(r,\xi)$ denote the polar coordinates centered at $p$. The following lemma gives the existence of a center of mass of $A$ in $M$ (see Theorem $1$ of \cite{S.07}).

\begin{lem} \label{lem:center of mass}
Let $G: [0, \infty)\rightarrow \mathbb{R}$ be a continuous function which is positive on $(0, \infty)$. Then there exists a point $p \in CA$ such that
\begin{align*}
\int_{A} G(\|X\|_{p}) X ds = 0.
\end{align*}
\end{lem}
Thus for a convex bounded domain $\Omega$ in a 
noncompact rank-$1$  symmetric space $(M,ds^{2})$, we can choose a point $p\in \Omega$ such that
\begin{align*}
 \int_{\partial \Omega} g(\| \exp_{p}^{-1}(q) \|_{p}) \frac {\langle\exp_{p}^{-1}(q), X_{i} \rangle}{\| \exp_{p}^{-1}(q) \|_{p}} dA  = 0, \quad 1 \leq i \leq kn,
\end{align*}
where $(X_{1},X_{2}, \ldots, X_{kn})$ is an orthonormal basis of $ T_{p}(M)$. Hereafter, we shall denote $r = r(q)= \| \exp_{p}^{-1}(q) \|_{p}$. Consider a $kn \times kn$ matrix $A = (a_{ij})$, where
\begin{align*}
    a_{ij} = \int_{\partial \Omega} g(r) \frac {\langle\exp_{p}^{-1}(q), X_{i} \rangle}{r} u_{j} dA.
\end{align*}
Using the QR-factorization theorem, we know that there exist an upper triangular matrix $T=(t_{ij})$ and an orthogonal matrix $U = (u_{ij})$ such that $T = UA$, i.e.,
\begin{align*}
   t_{ij}&= \sum_{m=1}^{kn} u_{im} a_{mj} \\
         &= \sum_{m=1}^{kn} u_{im} \int_{\partial \Omega} g(r) \frac {\langle\exp_{p}^{-1}(q), X_{m} \rangle}{r} u_{j} dA \\
         & = 0 \;\mbox{for} \; 1 \leq j<i \leq kn.
\end{align*}
Thus we can choose an orthogonal basis $(X_{1},X_{2}, \ldots, X_{kn})$ of $ T_{p}(M)$ such that
\begin{align}\label{int-stat}
  \int_{\partial \Omega} g(r) \frac {\langle\exp_{p}^{-1}(q),  X_{i} \rangle}{r} u_{j} dA = 0 \mbox{ for }  1 \leq j<i \leq kn.  
\end{align}
 
\section{Some properties of test functions}
 Let $\Omega$ be a convex bounded domain in a 
 noncompact rank-$1$  symmetric space $(M,ds^{2})$. Fix a point $p \in \Omega$. Then any $q \in \Omega$ can be represented as $q = \exp_{p} \left(r(q) w(q) \right)$, where $r(q) \in \mathbb{R}^{+}$ and $w(q)$ is a unit vector in $T_{p}M$. Thus $r(q) = \| \exp_{p}^{-1}(q) \|_{p}$ and $w(q) = \frac{\exp_{p}^{-1}(q)}{\| \exp_{p}^{-1}(q) \|_{p}}$. For a given orthonormal basis $(X_{1},X_{2}, \ldots, X_{kn})$ of $ T_{p}(M)$, let 
\begin{align*}
f_{i} (q) = \left\langle X_{i}, w(q) \right\rangle = \frac{\left\langle X_{i}, \exp_{p}^{-1}(q) \right\rangle}{\| \exp_{p}^{-1}(q) \|_{p}} = \frac{x_{i}}{r}.
\end{align*}
Let $(\eta_{1}, \eta_{2}, \ldots, \eta_{kn})$ be another orthonormal basis of $T_{p}(M)$ such that $\eta_{1} = w(q)$ and $\eta_{j+1} = J_{j} \eta_{1}$ for $j = 1, 2, \ldots, (k-1)$,  where $J_{j}$
are $k-1$ orthogonal complex structure on $M$
(\cite{CR.19}).
\noindent
Then using Lemma $4.11$ of \cite{CR.19}, we obtain 
\begin{align*}
\nabla f_{i} (q) = \sum_{j=2}^{k} \frac{\left\langle X_{i}, \eta_{j} \right\rangle }{\sinh r \cosh r} \eta_{j} + \sum_{j=k+1}^{kn} \frac{\left\langle X_{i}, \eta_{j} \right\rangle }{\sinh r} \eta_{j}.
\end{align*}

Now the following lemma follows easily.
\begin{lem} 
\begin{enumerate}
\item $\sum_{i = 1}^{kn}  \|\nabla \left(\frac{x_{i}}{r}\right)\|^{2} = \lambda_{1} (S(r))$.
\item $ \int_{\Omega}  \left(g'(r)\right)^{2} \frac{x_{i}^{2}}{r^{2}}  dV = \frac{1}{kn} \int_{\Omega}  \left(g'(r)\right)^{2} dV$.
\end{enumerate}
\end{lem}

\begin{lem}\label{grad} 
The following holds:
\begin{align*}
\sum_{i=1}^{kn} \frac{1}{\mu_{i}(\Omega)} \|\nabla^{S(r)} \left(\frac{x_{i}}{r}\right)\|^{2} \leq \frac{1}{l} \sum_{i=1}^{l} \frac{\lambda_{1} (S(r))}{\mu_{i}(\Omega)}.
\end{align*}
\end{lem}
\begin{proof}
\begin{align*}
\sum_{i=1}^{kn} \frac{1}{\mu_{i}(\Omega)} \Big\|\nabla^{S(r)} \left(\frac{x_{i}}{r}\right)\Big\|^{2}& \leq \sum_{i=1}^{l} \frac{1}{\mu_{i}(\Omega)} \Big\|\nabla^{S(r)} \left(\frac{x_{i}}{r}\right)\Big\|^{2} + \frac{1}{\mu_{l+1}(\Omega)} \left(   \lambda_{1} (S(r))- \sum_{i=1}^{l} \Big\|\nabla^{S(r)} \left(\frac{x_{i}}{r}\right)\Big\|^{2} \right) \\
& = \sum_{i=1}^{l} \frac{1}{\mu_{i}(\Omega)} \Big\|\nabla^{S(r)} \left(\frac{x_{i}}{r}\right)\Big\|^{2} + \frac{1}{\mu_{l+1}(\Omega)} \sum_{i=1}^{l} \left( \frac{\lambda_{1} (S(r))}{l}- \Big\|\nabla^{S(r)} \left(\frac{x_{i}}{r}\right)\Big\|^{2} \right) \\
& \leq \sum_{i=1}^{l} \frac{1}{\mu_{i}(\Omega)} \Big\|\nabla^{S(r)} \left(\frac{x_{i}}{r}\right)\Big\|^{2} + \sum_{i=1}^{l} \frac{1}{\mu_{i}(\Omega)} \left( \frac{\lambda_{1} (S(r))}{l}- \Big\|\nabla^{S(r)} \left(\frac{x_{i}}{r}\right)\Big\|^{2} \right) \\
& = \frac{1}{l} \sum_{i=1}^{l} \frac{\lambda_{1} (S(r))}{\mu_{i}(\Omega)}.
\end{align*}
\end{proof}

\noindent
The following Lemma is proved in \cite{BS.14} (Lemma $3.3$ (1)). 

\begin{lem} \label{lem: integral g}
 Let $g$ be the function defined as in \eqref{function g}. Then
\begin{align*}
 \int_{\partial\Omega} g^{2}(r) dA \geq \mbox{vol} \left(S(R)\right) g^{2}(R),
 \end{align*}
 and equality holds if and only if $\partial\Omega$ is a geodesic sphere of radius $R$ centered at $p$.
\end{lem}


\section{Proof of Theorem \ref{thm:main theorem}}
Let $\{u_{i}\}_{i=0}^{\infty}$ be an orthonormal sequence of eigenfunctions corresponding to eigenvalues $\{\mu_{i}(\Omega)\}_{i=0}^{\infty}$. 
In view of  (\ref{int-stat}),   
we may choose an orthonormal basis $(X_{1},X_{2}, \ldots, X_{kn})$ of $T_{p}(M)$ and corresponding normal coordinate system $(x_{1},x_{2}, \ldots, x_{kn})$ centered at a point $p$ such that 
\begin{align*} \label{eqn:test function}
\int_{\partial \Omega} g(r) \frac {\langle\exp_{p}^{-1}(q), X_{i} \rangle}{\| \exp_{p}^{-1}(q) \|_{p}} u_{j} dA = \int_{\partial\Omega} g(r) \frac {x_{i}}{r} u_{j} dA = \int_{\partial\Omega} g(r) \omega(\xi) u_{j} dA  = 0, \quad 0 \leq j \leq {i-1}.
\end{align*}

Using the variational characterization of $\mu_{i}(\Omega)$ (\ref{var-cha}), we get
\begin{align*}
\mu_{i}(\Omega) \int_{\partial\Omega} \left(g(r)\frac{x_{i}}{r}\right)^{2} dA & \leq \int_{\Omega} \Big\|\nabla  \left(g(r)\frac{x_{i}}{r}\right)\Big\|^{2} dV \\
& = \int_{\Omega} \left(g'(r)\right)^{2} \frac{x_{i}^{2}}{r^{2}} + g^{2}(r) \Big\|\nabla^{S(r)} \left(\frac{x_{i}}{r}\right)\Big\|^{2} dV.
\end{align*}

Dividing both the sides by $\mu_{i}(\Omega)$ and summing over $i$ gives,
\begin{align*}
 \int_{\partial\Omega} g^{2}(r) dA \leq \sum_{i=1}^{kn} \frac{1}{\mu_{i}(\Omega)}\int_{\Omega} \left(  \left(g'(r)\right)^{2} \frac{x_{i}^{2}}{r^{2}} + g^{2}(r) \Big\|\nabla^{S(r)} \left(\frac{x_{i}}{r}\right)\Big\|^{2} \right)  dV.
\end{align*}

Using Lemma \ref{grad}  we get,
\begin{align*}
 \int_{\partial\Omega} g^{2}(r) dA & \leq \frac{1}{kn} \sum_{i=1}^{kn} \frac{1}{\mu_{i}(\Omega)}\int_{\Omega} \left(g'(r)\right)^{2} dV + \frac{1}{l} \sum_{i=1}^{l} \frac{1}{\mu_{i}(\Omega)}\int_{\Omega} g^{2}(r) \lambda_{1}(S(r))  dV, \\
 & \leq \frac{1}{l} \sum_{i=1}^{l} \frac{1}{\mu_{i}(\Omega)}\int_{\Omega} \left( \left(g'(r)\right)^{2} + g^{2}(r) \lambda_{1}(S(r))\right)  dV.
\end{align*}
Thus we have,
\begin{align*}
\frac{1}{l} \sum_{i=1}^{l} \frac{1}{\mu_{i}(\Omega)} & \geq \frac{\int_{\partial\Omega} g^{2}(r) dA}{\int_{\Omega} \left( \left(g'(r)\right)^{2} + g^{2}(r) \lambda_{1}(S(r))\right)  dV}, \\ 
& \geq \frac{\int_{\partial B(r)} g^{2}(r) dA}{\int_{B(r)} \left( \left(g'(r)\right)^{2} + g^{2}(r) \lambda_{1}(S(r))\right)  dV},\\
&= \frac{1}{\mu_{1}(B(r))},\\
 &= \frac{1}{l} \sum_{i=1}^{l} \frac{1}{\mu_{i}(B(r))}.
\end{align*}
This proves Theorem \ref{thm:main theorem}.
\begin{rmk}
Harmonic manifolds are the generalization of rank one symmetric spaces. 
In \cite{RR.97}, it is shown that the volume density function characterizes the hyperbolic space, the complex hyperbolic space, and the quaternoinic hyperbolic space. It is natural to ask: if $(M, g)$ is a harmonic Riemannian manifold, is Theorem \ref{thm:main theorem} true?
\end{rmk}

\end{document}